\documentclass[11pt,reqno,letterpaper]{amsart}

\usepackage{amssymb}
\usepackage{amsmath}
\usepackage{amsthm}
\usepackage{bbm}
\usepackage{mathrsfs}
\usepackage{doi}
\usepackage{graphicx}
\usepackage[shortlabels]{enumitem}

\usepackage{bookmark}
\usepackage{hyperref}
\hypersetup{pdfstartview={FitH}}

\numberwithin{equation}{section}

\newtheorem{theorem}{Theorem}
\newtheorem{lemma}[theorem]{Lemma}
\newtheorem{proposition}[theorem]{Proposition}
\newtheorem{corollary}[theorem]{Corollary}

\theoremstyle{definition}
\newtheorem{definition}{Definition}
\newtheorem{remark}[theorem]{Remark}

\begin{document}

\title[Lacunary sequences representing rational numbers]{Lacunary sequences whose reciprocal sums represent all rational numbers in an interval}

\author[Wouter van Doorn]{Wouter van Doorn}
\address{Wouter van Doorn, Groningen, the Netherlands} 
\email{wonterman1@hotmail.com}

\author[Vjekoslav Kova\v{c}]{Vjekoslav Kova\v{c}}
\address{Vjekoslav Kova\v{c}, Department of Mathematics, Faculty of Science, University of Zagreb, Bijeni\v{c}ka cesta 30, 10000 Zagreb, Croatia}
\email{vjekovac@math.hr}

 %Convergence and divergence of series and sequences

%\date{\today}


\begin{abstract}
Disproving a conjecture of Bleicher and Erd\H{o}s, we show that there exists a lacunary sequence of positive integers such that finite sums of reciprocals of its terms attain all rational numbers from a non-empty open interval. We also study several stronger variants of their original problem: determining the value of the optimal lacunarity parameter, representing rational numbers infinitely many times, finding such lacunary sequences with arbitrarily large jumps, and relating the maximal length of a filled interval to a prescribed lacunarity parameter.
\end{abstract}

\maketitle

%%%%%%%%%%%%%%%%%%%%%%%%%%%%%%%%%%%%%%%%%%%%%%%%

\section{Introduction}

In \cite{BE76} and the follow-up work \cite{BE76a}, Bleicher and Erd\H{o}s studied various properties of representations of rational numbers as \emph{Egyptian fractions}, which are finite sums of the form
\[ \frac{1}{n_1} + \frac{1}{n_2} + \cdots + \frac{1}{n_t} \]
for some positive integers $n_1<n_2<\cdots<n_t$.
In \cite[Conjecture 4]{BE76} they asked the following.
\begin{quote}
\emph{Let $n_1<n_2<\cdots$ be an infinite sequence of positive integers such that $n_{i+1}/n_i>c>1$. Can the set of rational numbers $a/b$ for which
\begin{equation}\label{eq:BEsums}
\frac{a}{b} = \frac{1}{n_{i_1}} + \frac{1}{n_{i_2}} + \cdots + \frac{1}{n_{i_t}}
\end{equation}
is solvable for some $t$ contain all the rational numbers in some interval $(\alpha,\beta)$? We conjecture not.}\footnote{It is understood from the context that the indices $i_1,i_2,\ldots,i_t$ are distinct.}
\end{quote}
This question also appeared in Erd\H{o}s and Graham's problem book \cite[p.~58]{EG80} and, more recently, as Problem \#355 on Bloom's website \emph{Erd\H{o}s problems} \cite{EP}.

Erd\H{o}s and Graham \cite{EG80} suggested that a potential motivation to ask the above question comes from the notion of a \emph{reciprocal basis for the rational numbers} (introduced by Wilf in \cite{W61}), which is a set of positive integers for which every positive rational number can be written as a finite sum of reciprocals of its elements. For example, Graham \cite[\S4]{G64} showed that any set containing all large enough primes and squares is such a reciprocal basis, while Erd\H{o}s and Stein \cite[Theorem 4]{ES63} showed that for any positive sequence $(s_i)_{i=1}^{\infty}$ such that $\sum_i 1/s_i$ diverges, a reciprocal basis $(n_i)_{i=1}^{\infty}$ exists satisfying $n_i\geqslant  s_i$ for all $i$.

\begin{definition}
A sequence $(n_i)_{i=1}^{\infty}$ satisfying $n_{i+1}/n_i\geqslant  \lambda$ for some parameter $\lambda>1$ and every index $i\in\mathbb{N}$ is said to be \emph{$\lambda$-lacunary}. It is simply said to be \emph{lacunary} if it is $\lambda$-lacunary for some $\lambda>1$.
\end{definition}

An argument that seems to support the conjecture by Bleicher and Erd\H{o}s is that, even if one only wants to represent fractions whose denominators are powers of $2$, one cannot do better than a $2$-lacunary sequence, i.e., one cannot use a $\lambda$-lacunary sequence with $\lambda>2$. This is similarly true for the powers of any prime, and it is not obvious whether all these sequences can be efficiently combined into a single sequence which is still lacunary. Moreover, the aforementioned results from \cite{ES63,G64c} are about sequences $(n_i)_{i=1}^{\infty}$ that grow relatively slowly. This may have strengthened the belief of Bleicher and Erd\H{o}s that exponential growth is not possible. Somewhat cryptically, they commented:
\begin{quote}
\emph{If this conjecture is true then according to Graham \cite{G64} this is best possible.} \cite[p.\,167]{BE76}
\end{quote}

We will come back to Graham's results from \cite{G64} shortly. In any case, perhaps rather surprisingly, we show that lacunary sequences whose reciprocal sums represent all rational numbers in an interval in fact do exist. 
For any sequence of positive reals $(x_i)_{i=1}^{\infty}$, Erd\H{o}s and Graham \cite[p.~53]{EG80} denoted\footnote{The letter $P$ probably stands for the \emph{parallelotope} spanned by infinitely many ``edges'' $x_1,x_2,\ldots$.}
\begin{equation}\label{eq:theset}
P((x_i)_{i=1}^{\infty}) := \Big\{ \sum_{i\in T} x_i \,:\, T\subset\mathbb{N} \text{ finite} \Big\},
\end{equation}
so that the set of all finite sums in \eqref{eq:BEsums} can be written as $P((1/n_i)_{i=1}^{\infty})$. 

\begin{theorem}\label{thm01}
\begin{enumerate}[(a)]
\item\label{ita}
For every $\lambda\in(1,2)$ there exists a $\lambda$-lacunary sequence of positive integers $n_1,n_2,n_3,\ldots$ such that $P((1/n_i)_{i=1}^{\infty})\supseteq[0,2]\cap\mathbb{Q}$.

\item\label{itb}
The sequence from part \ref{ita} can be constructed so that it also satisfies $\lim_{i\to\infty}n_{i+1}/n_i=2$ and that every rational number in $(0,2]$ is equal to the sum $\sum_{i\in T}1/n_i$ for infinitely many finite sets $T\subset\mathbb{N}$.

\item\label{itc}
There is no $2$-lacunary sequence of positive integers $n_1,n_2,n_3,\ldots$ such that $P((1/n_i)_{i=1}^{\infty})$ contains all rational numbers from a non-empty open interval.

\end{enumerate}

\end{theorem}

The proof of parts \ref{ita} and \ref{itb} of Theorem \ref{thm01} is more or less constructive, and even provides such a sequence for which the equality $n_{i+1}/n_i = 2$ holds for infinitely many indices $i$. It will be presented in Section \ref{sec:thm01} after describing, in Section \ref{sec:general_suff}, a more general sufficient condition for a sequence of positive integers to have reciprocal sums that contain an interval; see Proposition \ref{prop:general}.
Already part \ref{ita} answers (a strong version of) the question of Bleicher and Erd\H{o}s, while part \ref{itb} achieves further strengthenings.
Part \ref{itc} shows that the assumption $\lambda<2$ in part \ref{ita} is optimal. It will be the content of Corollary \ref{cor:necess1} in Section \ref{sec:general_nec}.

Presence of the explicit interval $[0,2]$ in Theorem \ref{thm01} motivates the following definition.

\begin{definition}
For every $\lambda\in(1,\infty)$ let $R(\lambda)$ be the least upper bound on the length $\beta-\alpha$ of an interval $(\alpha,\beta)$ for which there exists a $\lambda$-lacunary sequence of positive integers $n_1,n_2,n_3,\ldots$ such that $P((1/n_i)_{i=1}^{\infty})$ contains all rational numbers from $(\alpha,\beta)$. 
\end{definition}

We can give a formula for $R(\lambda)$, which, in particular, allows one to determine its numerical value to arbitrary precision.

\begin{theorem}\label{thm02}
For every $\lambda\in(1,2)$ the quantity $R(\lambda)$ is given by
$\sum_{i=1}^{\infty}1/a_i$,
where the sequence $(a_i)_{i=1}^{\infty}$ is recursively defined as
\[ \left\{ {\setlength{\arraycolsep}{2pt}\begin{array}{rl}
a_1 & := 1, \\
a_{i+1} & := \lceil \lambda a_i \rceil \quad\text{for } i\geqslant  1.
\end{array}} \right. \]
Moreover,
\[ \lim_{\lambda\to 1^+} R(\lambda) = +\infty,\quad \lim_{\lambda\to 2^-} R(\lambda) = 2, \]
and $R(\lambda)=0$ when $\lambda\geqslant  2$.
\end{theorem}

\begin{figure}
\includegraphics[width=0.5\linewidth]{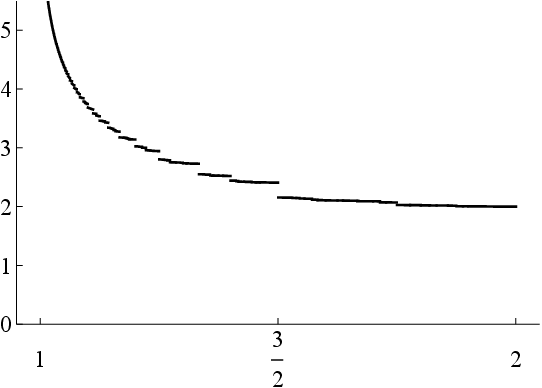}
\caption{Graph of the function $R$ on $(1,2)$.}
\label{fig:fig1}
\end{figure}

A graph of the function $\lambda\mapsto R(\lambda)$ on the interval $(1,2)$ is depicted in Figure \ref{fig:fig1}.
The proof of Theorem \ref{thm02} will be presented in Section \ref{sec:thm02}, but it will again be a consequence of general observations from Section \ref{sec:general_suff}. Both Sections \ref{sec:thm01} and \ref{sec:thm02} will end with some explicit examples.

As it turns out, the proof of Theorem \ref{thm02} gives, for every $\varepsilon>0$, a $\lambda$-lacunary sequence $(n_i)_{i=1}^{\infty}$ such that $P((1/n_i)_{i=1}^{\infty})\supseteq(0,R(\lambda)-\varepsilon)\cap\mathbb{Q}$. That is, it is more economical to fill in $(0,R(\lambda)-\varepsilon)\cap\mathbb{Q}$ rather than $(\alpha,\alpha+R(\lambda)-\varepsilon)\cap\mathbb{Q}$ for some $\alpha>0$.

The approach we take in order to prove the above results is rather flexible, and one may even require the sequence $n_1,n_2,\ldots$ to have infinitely many large ratios $n_{i+1}/n_i$ in the following precise sense.

\begin{theorem}\label{thm03}
For every $\Lambda\geqslant 2$ and $1<\lambda<\Lambda/(\Lambda-1)$ there exists a $\lambda$-lacunary sequence of positive integers $(n_i)_{i=1}^{\infty}$ such that
\begin{equation}\label{eq:largejumps}
n_{i+1}>\Lambda n_i \quad\text{for infinitely many indices } i\in\mathbb{N}, 
\end{equation}
and for which the set $P((1/n_i)_{i=1}^{\infty})$ contains all rational numbers from the interval $[0,\sum_{i=1}^{\infty}1/n_i)$.
\end{theorem}

Even though property \eqref{eq:largejumps} might appear to collide with part \ref{itc} of Theorem \ref{thm01}, we can make up for ``large jumps'' by following them up with a large number of small ones, while still maintaining $\lambda$-lacunarity, as long as $\lambda<\Lambda/(\Lambda-1)$. After the proof of Theorem \ref{thm03}, which can be found in Section \ref{sec:thm03}, we will furthermore explain why this latter inequality is optimal; no such $\lambda$-lacunary sequence exists with $\lambda\geqslant \Lambda/(\Lambda-1)$.

Let us mention that we were not able to fully characterize all sequences $n_1<n_2<\cdots$ of positive integers such that $P((1/n_i)_{i=1}^{\infty})$ consists of all rational numbers contained in $[0,\sum_{i=1}^{\infty}1/n_i)$. Such a characterization seems to be out of reach of the ideas employed here. However, Inequality \eqref{eq:neceint} below is a necessary requirement, while Proposition \ref{prop:general} will provide us with sufficient conditions. For sequences of a special form defined below, Graham \cite{G64} actually did achieve one representability characterization, while Eppstein \cite{Epp21} recently proved a result of a similar nature. As these results notably relate to the approach we take in this paper, let us take a more detailed look at them.

In \cite[Theorem 1]{Epp21} Eppstein showed that, if a subset $S$ of positive integers 
\begin{itemize}
\item is closed under multiplication by $2$ (i.e., $2S\subseteq S$) and
\item contains a multiple of each (odd) integer,
\end{itemize}
then $P((1/n)_{n\in S})$ is precisely $[0,\sum_{n\in S}1/n)\cap\mathbb{Q}$.\footnote{Here we are slightly abusing the notation and writing $P((x_n)_{n\in S})$ in the place of $P((x_{n_i})_{i=1}^{\infty})$, where $(n_i)_{i=1}^\infty$ is the strictly increasing sequence enumerating a set $S\subseteq\mathbb{N}$.}
He then applied this result to resolve a conjecture of Sun \cite{Sun}, which states that every positive rational number can be written as a finite sum of unit fractions whose denominators are distinct practical numbers, which are positive integers $N$ such that every positive integer smaller than $N$ can be written as a sum of distinct divisors of $N$.

Even though it seems that Eppstein's theorem cannot directly be applied to answer the aforementioned question of Bleicher and Erd\H{o}s, in its proof we do recognize similarity with some of the techniques employed here. In particular, using these techniques we are able to present a short alternative proof of Eppstein's theorem (see Theorem \ref{thm04} in the appendix).

As for Graham's theorem, for a set $S\subseteq\mathbb{N}$ enumerated by an increasing sequence $(s_i)_{i=1}^{\infty}$, consider the set of all finite products of distinct elements of $S$,
\[ \{ s_{i_1} s_{i_2} \cdots s_{i_\ell} \,:\, \ell\in\mathbb{N},\ i_1<i_2<\cdots<i_\ell \}, \]
thus removing repeated values, and let $M(S)$ denote the sequence obtained by sorting its elements in the ascending order. Assuming that 
\begin{itemize}
\item the sequence of ratios $(s_{i+1}/s_i)_{i=1}^{\infty}$ is bounded and that 
\item $P(M(S))$ contains all sufficiently large positive integers,
\end{itemize}
Graham showed \cite[Theorem 5]{G64} that a rational number belongs to $P(1/M(S))$ if and only if\footnote{Here we write $1/(n_i)_{i=1}^\infty$ to denote $(1/n_i)_{i=1}^{\infty}$.} 
\begin{itemize}
\item it is a decreasing limit of a sequence in $P(1/M(S))$ and 
\item its denominator (in least terms) divides some term of $M(S)$.
\end{itemize}
Graham wondered if the condition on the boundedness of consecutive ratios $s_{i+1}/s_i$
\begin{quote}
(\ldots) \emph{could be weakened or even removed} \cite[p.\,205]{G64}
\end{quote}
and he similarly mentioned the following some $50$ years later:
\begin{quote}
\emph{It is not known whether the condition that $s_{n+1}/s_n$ be bounded is needed for the conclusion of the theorem to hold.} \cite[p.\,295]{Gra13}
\end{quote}
Let us record a simple observation that Eppstein's theorem already answers Graham's question in a certain sense, by giving an example of a set $S$ with $\limsup_{i\to\infty}s_{i+1}/s_i=\infty$ for which the same conclusion nevertheless holds. 
Namely, if one takes
\[ S = \big\{ 2^{2^k} \,:\, k\in\mathbb{N} \big\} \cup \big\{ k! \,:\, k\in\mathbb{N} \setminus \{1\} \big\}, \]
then one readily verifies the containment $2M(S)\subseteq M(S)$, so that Eppstein's theorem applies to the set $M(S)$.
Thus, the above $S$ provides an example where the conclusion of Graham's theorem holds, even though neither of its two conditions are satisfied. Indeed, the ratio of consecutive values in $S$ is unbounded, while $S$, and therefore $P(M(S))$, only contains even integers.

As a final note before we dive in, it will be useful to relate the set defined in \eqref{eq:theset} to the infinitary (and uncountable) notion of the so-called achievement set, where we allow subsets $S \subseteq \mathbb{N}$ to range over infinite sets as well. An overview of the classical research on achievement sets of series can be found in the expository paper \cite{BFP13}. Variations of this concept have proved to be instrumental in recent solutions to some other problems on unit fractions posed by Erd\H{o}s, see for example \cite{CK25} and \cite{KT25}. These latter papers are concerned with infinite sums of unit fractions however, which can no longer be viewed as Egyptian fractions. Therefore, these connections are useful mainly to motivate subsequent formulations and proofs; no results from \cite{CK25} or \cite{KT25} will be needed here.

%%%%%

\section{General necessary conditions}
\label{sec:general_nec}

Let us begin by investigating necessary conditions for the set defined in \eqref{eq:theset} to contain all rational numbers from a non-degenerate open interval.
Keeping in mind the intended substitution $x_i=1/n_i$ for a lacunary sequence of positive integers $(n_i)_{i=1}^{\infty}$, let us first treat $(x_i)_{i=1}^{\infty}$ as just any decreasing sequence of positive real numbers for which $\sum_i x_i$ converges. Consider the set of both finite and infinite sums of its terms:
\[ A((x_i)_{i=1}^{\infty}) := \Big\{ \sum_{i\in T} x_i \,:\, T\subseteq\mathbb{N} \Big\}. \]
This set is often called the \emph{achievement set} of the series $\sum_i x_i$. Already Kakeya \cite{Kak14a,Kak14b} observed that the condition
\[ x_i > \sum_{j=i+1}^{\infty} x_j \quad \text{for every } i\in\mathbb{N} \]
implies that $A((x_i)_{i=1}^{\infty})$ is a Cantor set in the topological sense. In particular, it is a compact set with empty interior.
This is not too hard to see and we can, in fact, reach the same conclusion from a slightly weaker hypothesis.

\begin{lemma}\label{lm:Kakeya}
If $x_1>x_2>x_3>\cdots$ are positive reals such that $\sum_i x_i$ converges and that also satisfy
\begin{align}
& x_i \geqslant  \sum_{j=i+1}^{\infty} x_j \quad \text{for every } i\in\mathbb{N} \text{ and } \label{eq:Cantor0} \\
& x_i > \sum_{j=i+1}^{\infty} x_j \quad \text{for infinitely many } i\in\mathbb{N},  \label{eq:Cantor1}
\end{align}
then the set $A((x_i)_{i=1}^{\infty})$ is compact and has empty interior.
\end{lemma}

\begin{proof}
For compactness of $A((x_i)_{i=1}^{\infty})$ one only needs the (absolute) convergence of the series $\sum_i x_i$; see \cite{Kak14a,Kak14b} or the more recent references mentioned after \cite[Theorem 21.4]{BFP13}. Let us turn to the claim about its interior.
If $i_1<i_2<i_3<\cdots$ are precisely the indices $i$ for which \eqref{eq:Cantor1} holds, then, for every $k\in\mathbb{N}$, the set $A((x_i)_{i=1}^{\infty})$ is contained in the union of closed intervals
\begin{equation*} 
\mathcal{U} := \bigcup_{\epsilon_1,\epsilon_2,\ldots,\epsilon_{i_k}\in\{0,1\}} \Big[\sum_{j=1}^{i_k} \epsilon_j x_j , \sum_{j=1}^{i_k} \epsilon_j x_j + r_k\Big], 
\end{equation*}
each of which has length $r_k:=\sum_{j=i_k+1}^{\infty}x_j$.
These intervals have disjoint interiors and appear on the real line in the increasing order with respect to the lexicographical order on the tuples $(\epsilon_1,\ldots,\epsilon_{i_k})$ from $\{0,1\}^{i_k}$. Namely, if $(\epsilon_1,\ldots,\epsilon_{i_k})$ precedes $(\epsilon'_1,\ldots,\epsilon'_{i_k})$ and $1\leqslant \ell\leqslant i_k$ is the smallest index at which these tuples differ (so that $0=\epsilon_\ell<\epsilon'_\ell=1$), then the intervals
{\allowdisplaybreaks
\begin{align*}
I & := \Big[\sum_{j=1}^{i_k} \epsilon_j x_j, \sum_{j=1}^{i_k} \epsilon_j x_j + r_k\Big], \\
I' & := \Big[\sum_{j=1}^{i_k} \epsilon'_j x_j , \sum_{j=1}^{i_k} \epsilon'_j x_j + r_k\Big]
\end{align*}
}
satisfy
{\allowdisplaybreaks
\begin{align*}
\max I & = \sum_{j=1}^{i_k} \epsilon_j x_j + r_k 
\leqslant \sum_{j=1}^{\ell-1} \epsilon_j x_j + \sum_{j=\ell+1}^{i_k} x_j + r_k \\
& = \sum_{j=1}^{\ell-1} \epsilon_j x_j + \sum_{j=\ell+1}^{\infty} x_j 
\stackrel{\eqref{eq:Cantor0}}{\leqslant} \sum_{j=1}^{\ell-1} \epsilon_j x_j + x_\ell \\
& = \sum_{j=1}^{\ell-1} \epsilon'_j x_j + x_\ell 
= \sum_{j=1}^{\ell} \epsilon'_j x_j 
\leqslant \sum_{j=1}^{i_k} \epsilon'_j x_j 
= \min I',
\end{align*}
}
so $I$ is to the left of $I'$ and they can only touch at the endpoints.
The inequality is strict when $\ell=i_k$, since we can then apply \eqref{eq:Cantor1} instead of \eqref{eq:Cantor0}.
When $\ell<i_k$, the inequality is also strict unless the two tuples determining $I$ and $I'$ respectively read
\begin{align*}
& (\epsilon_1,\ldots,\epsilon_{\ell-1},0,1,\ldots,1), \\
& (\epsilon_1,\ldots,\epsilon_{\ell-1},1,\underbrace{0,\ldots,0}_{i_k-\ell}).
\end{align*}
Consequently, each interval from the union $\mathcal{U}$ can touch at most one other interval. We conclude that $A((x_i)_{i=1}^{\infty})$ is, in fact, contained in the union of disjoint closed intervals of length at most $2r_k$.
Since $r_k$ tends to $0$ as $k\to\infty$ (due to the convergence of $\sum_j x_j$), we see that the achievement set cannot contain a non-degenerate interval.
\end{proof}

We would like to point out that condition \eqref{eq:Cantor1} by itself is not sufficient to deduce that the achievement set has empty interior, as the so-called \emph{Cantorval} does have non-empty interior; see the examples constructed in \cite{WS80,GN88}, noting that this suggestive name was invented later.
The Guthrie--Nymann Cantorval \cite{GN88} is defined as the achievement set $A((x_i)_{i=1}^{\infty})$, where
\[ x_i := \begin{cases}
\displaystyle\frac{3}{2^{i+1}} & \text{for $i$ odd},\\[2mm]
\displaystyle\frac{1}{2^{i-1}} & \text{for $i$ even},
\end{cases} \]
while a general Cantorval is then defined as every set homeomorphic to it.

The previous lemma applies in particular to the reciprocals of $2$-lacunary sequences.

\begin{corollary}\label{cor:necess1} 
Finite sums of reciprocal terms of a $2$-lacunary sequence $n_1,n_2,n_3,\ldots$ cannot attain all rational numbers from a non-empty open interval.
\end{corollary}

\begin{proof}
Let us specialize Lemma \ref{lm:Kakeya} to $x_i=1/n_i$.
The $2$-lacunarity of $(n_i)_{i=1}^{\infty}$ gives 
\[ x_j = \frac{1}{n_j} \leqslant \frac{1}{2^{j-i}n_i} = 2^{i-j} x_i \]
for all $j>i$, so that
\begin{equation}\label{eq:Cantor2}
\sum_{j=i+1}^{\infty} x_j \leqslant \sum_{j=i+1}^{\infty} 2^{i-j} x_i = x_i
\end{equation}
for every index $i$, verifying \eqref{eq:Cantor0}.
If \eqref{eq:Cantor1} is satisfied as well, then we claim that the corresponding set of finite sums $P((x_i)_{i=1}^{\infty})$ cannot contain all rational numbers from a non-degenerate open interval $(\alpha,\beta)$. Indeed, by the density of the rational numbers, the closed set $A((x_i)_{i=1}^{\infty})\supset P((x_i)_{i=1}^{\infty})$ would then contain $[\alpha,\beta]$, contradicting Lemma \ref{lm:Kakeya}.
Thus, suppose that \eqref{eq:Cantor1} fails.
This implies that we must actually have equality in \eqref{eq:Cantor2} for all sufficiently large indices $i$, say all $i \geqslant i_0$. This in turn implies that $n_j$ is exactly equal to $2^{j-i_0}n_{i_0}$ for all $j \geqslant i_0$. In particular, for all primes $p > n_{i_0}$, $p$ does not divide any element $n_i$ of the sequence. We therefore see that $P((x_i)_{i=1}^{\infty})$ does not contain any fraction whose denominator is divisible by a prime larger than $n_{i_0}$, so that it certainly does not contain all rational numbers from a non-empty open interval.
\end{proof}

We remark that Corollary \ref{cor:necess1} already confirms part \ref{itc} of Theorem \ref{thm01}, which turned out to be the easiest part of our main result.

Let us return to a general sequence $x_1>x_2>x_3>\cdots>0$ such that $\sum_{i=1}^{\infty} x_i<\infty$. Kakeya \cite{Kak14a,Kak14b} furthermore observed that $A((x_i)_{i=1}^{\infty})$ consists of a single real-line segment if and only if
\[ x_i \leqslant \sum_{j=i+1}^{\infty} x_j \quad \text{for every } i\in\mathbb{N}; \]
see \cite{BFP13}.
Moreover, when the latter condition is satisfied, the achievement set is precisely the ``largest possible'' interval, i.e., $[0,\sum_{i=1}^{\infty} x_i]$.
The necessity is, this time, especially simple: if $x_k>\sum_{j=k+1}^{\infty}x_j$ for some index $k$, then $A((x_i)_{i=1}^{\infty})$ has a ``gap'' $(\sum_{j=k+1}^{\infty}x_j,x_k)$ and it definitely cannot be an interval.
Specializing this to $x_i=1/n_i$ and recalling that the rational numbers are dense immediately yields the following corollary. 

\begin{corollary}\label{cor:necess2}
If $(n_i)_{i=1}^{\infty}$ is a strictly increasing sequence of positive integers such that $\sum_{i=1}^{\infty}1/n_i<\infty$ and the associated set $P((1/n_i)_{i=1}^{\infty})$ fills in all rational points from the interval
$[0,\sum_{i=1}^{\infty}1/n_i)$,
then it necessarily satisfies
\begin{equation}\label{eq:neceint}
\frac{1}{n_i} \leqslant \sum_{j=i+1}^{\infty} \frac{1}{n_j} \quad \text{for every } i\in\mathbb{N}.
\end{equation}
\end{corollary}

%%%%%

\section{General sufficient conditions}
\label{sec:general_suff}

Inequality \eqref{eq:neceint} can also motivate the formulation of sufficient conditions on $(n_i)_{i=1}^{\infty}$ such that the sums from \eqref{eq:theset} attain all rational values from some non-empty open interval.
Just as in the previous section, we begin with an observation which is valid for all real numbers (not just rational ones), where we would furthermore like to point out the resemblance to some of Graham's results, in particular \cite[Lemma 1]{G64b} and \cite[Lemma 1]{G64}.

\begin{lemma}\label{lm:reals}
Suppose that $x_1>x_2>\cdots>x_m>0$ are real numbers satisfying 
\[ x_i \leqslant \sum_{j=i+1}^{m}x_j + x_m \]
for all $i$ with $1 \le i \le m-1$. Then the set
\[ \Big\{ \sum_{i\in T}x_i : T\subseteq\{1,2,\ldots,m\} \Big\} \]
$x_m$-densely fills in the segment
$[0,\sum_{i=1}^{m}x_i]$,
i.e., it divides this segment into sub-intervals of length at most $x_m$.
\end{lemma}

\begin{proof}
We use induction on $m\in\mathbb{N}$. The base case $m=1$ is trivial because the segment is just $[0,x_1]$ in that case.
So assume $m\geqslant 2$, and let us apply the induction hypothesis to the sequence $x_2>\cdots>x_m$. We then conclude that
\begin{itemize}
\item the sums $\sum_{i\in T}x_i$, $T\subseteq\{2,\ldots,m\}$, $x_m$-densely fill in $[0,\sum_{i=2}^{m}x_i]$ and
\item the sums $x_1+\sum_{i\in T}x_i$, $T\subseteq\{2,\ldots,m\}$, $x_m$-densely fill in $[x_1,x_1+\sum_{i=2}^{m}x_i]$.
\end{itemize}
These two intervals either overlap or they are separated by at most $x_m$, as we assumed $x_1 \leqslant \sum_{i=2}^{m}x_i + x_m$. This completes the induction step.
\end{proof}

We are now in a position to formulate rather general sufficient conditions enabling the set\linebreak $P((1/n_i)_{i=1}^{\infty})$ to contain all rational numbers from $[0,\sum_{i=1}^{\infty}1/n_i)$. 

\begin{proposition}\label{prop:general}
Suppose that $n_1<n_2<n_3<\cdots$ is a sequence of positive integers such that
\begin{enumerate}[(1)]
\item\label{it2} every positive integer is a divisor of $n_i$ for some $i\in\mathbb{N}$.
\end{enumerate}
Also let $m_1<m_2<m_3<\cdots$ be an infinite sequence of distinguished indices such that
\begin{enumerate}[resume*]
\item\label{it3} for every $k\in\mathbb{N}$ the integer $n_{m_k}$ is divisible by all $n_j$, $1\leqslant j<m_k$, and
\item\label{it4} one has the inequality\footnote{Here $m_0\leqslant i$ should be interpreted as a void condition.}
\begin{equation}\label{eq:gensuff}
\frac{1}{n_i} \leqslant \sum_{j=i+1}^{m_k} \frac{1}{n_j} + \frac{1}{n_{m_k}} \quad\text{for all }k\in\mathbb{N}\text{ and }m_{k-1}\leqslant i<m_k.
\end{equation}
\end{enumerate}
Then $P((1/n_i)_{i=1}^{\infty})=[0,\sum_{i=1}^{\infty}1/n_i)\cap\mathbb{Q}$.
If, additionally,
\begin{equation}\label{eq:gensuff1}
\frac{1}{n_{m_k}} < \sum_{j=m_k+1}^{\infty} \frac{1}{n_j} \quad\text{holds for infinitely many }k\in\mathbb{N},
\end{equation}
then every rational number from the open interval $(0,\sum_{i=1}^{\infty}1/n_i)$ equals $\sum_{i\in T}1/n_i$ for infinitely many finite sets $T\subset\mathbb{N}$.
\end{proposition}

Let us remark that we allow the series $\sum_i 1/n_i$ to diverge in which case its sum and the sums appearing in \eqref{eq:gensuff1} need to be interpreted as $\infty$.

\begin{proof}
Assume that \ref{it2}--\ref{it4} hold and let us prove the first claim.
Take a rational number $0\leqslant q<\sum_{i=1}^{\infty}1/n_i$. By the assumptions \ref{it2} and \ref{it3}, there exists a sufficiently large $K\in\mathbb{N}$ such that the denominator of $q$ (in its least terms) divides $n_{m_K}$, while $q\leqslant \sum_{i=1}^{m_K}1/n_i$.
The idea is to apply Lemma \ref{lm:reals} to $m:=m_K$ and the unit fractions $x_i:=1/n_i$, $i=1,\ldots,m$.
To see that the inequality from Lemma \ref{lm:reals} is satisfied, take any $1\leqslant i<m$ and let $k\leqslant K$ be such that $m_{k-1}\leqslant i<m_k$. Condition \eqref{eq:gensuff} then gives
\[ x_i \leqslant \sum_{j=i+1}^{m_k} x_j + x_{m_k} \]
and
\[ x_{m_\ell} \leqslant \sum_{j=m_\ell+1}^{m_{\ell+1}} x_j + x_{m_{\ell+1}} \]
for $\ell=k,k+1,\ldots,K-1$. Summing all these inequalities yields
\[ x_i \leqslant \sum_{j=i+1}^{m} x_j + x_m, \]
which is the desired inequality.
On the one hand, Lemma \ref{lm:reals} now claims that the finite sums
\[ \sum_{i\in T} \frac{1}{n_i}, \quad T\subseteq\{1,2,\ldots,m\}, \]
$(1/n_m)$-densely fill in the segment $[0,\sum_{i=1}^{m}1/n_i]$. On the other hand, by assumption \ref{it3}, all these sums are integer multiples of $1/n_m$.
We conclude that the sums actually fill in the whole set $[0,\sum_{i=1}^{m}1/n_i]\cap(1/n_m)\mathbb{Z}$ and, in particular, at least one of them is equal to $q$.

Now we additionally assume \eqref{eq:gensuff1} and turn to the second claim. Take a rational number $0<q<\sum_{i=1}^{\infty}1/n_i$ and note that it suffices to represent it as $\sum_{i\in T}1/n_i$, where $\max T$ is larger than any prescribed $N\in\mathbb{N}$. Afterwards we obtain infinitely many representations of $q$ by increasing $N$. Take an index $k\in\mathbb{N}$ satisfying \eqref{eq:gensuff1} which is sufficiently large such that $m_k\geqslant  N$,
\[ \frac{1}{n_{m_k}} \leqslant q\leqslant \sum_{i=1}^{m_k}\frac{1}{n_i}, \]
and for which the denominator of $q$ divides $n_{m_k}$. Denote $m=m_k$.
Precisely the same reasoning as in the previous part of the proof, but applied to the rational number $q-1/n_m$, gives
\[ q - \frac{1}{n_m} = \sum_{i\in T'} \frac{1}{n_i} \]
for some $T'\subseteq\{1,2,\ldots,m\}$.
Now, \eqref{eq:gensuff1} and our choice of $k$ guarantee
$1/n_m < \sum_{i=1}^{\infty} 1/n_{m+i}$, so applying the first claim to the sequence
\[ n_{m+1} < n_{m+2} < \cdots \]
and the distinguished indices
\[ m_{k+1} < m_{k+2} < \cdots \]
yields the representation
\[ \frac{1}{n_m} = \sum_{i\in T''} \frac{1}{n_i} \]
for some finite $T''\subseteq\{m+1,m+2,\ldots\}$.
Adding the last two representations we obtain
\[ q = \sum_{i\in T'\cup T''} \frac{1}{n_i} \]
and it remains to note that $\max(T'\cup T'')>m\geqslant N$.
\end{proof}

\begin{remark}\label{rem:cond}
For the purpose of the proof of Theorem \ref{thm01} that will be given in the next section, observe that condition \eqref{eq:gensuff} is automatically satisfied for sequences $n_1<n_2<n_3<\cdots$ such that
\begin{equation}\label{eq:gensuff2}
n_{i+1}\leqslant 2n_i \quad\text{for every } i\in\mathbb{N}. 
\end{equation}
Indeed,
\[ \sum_{j=i+1}^{m_k} \frac{1}{n_j} + \frac{1}{n_{m_k}} 
\geqslant  \frac{1}{n_i} \Big( \sum_{j=i+1}^{m_k} \frac{1}{2^{j-i}} + \frac{1}{2^{m_k-i}} \Big) = \frac{1}{n_i}. \]
If furthermore the strict inequality
\begin{equation}\label{eq:gensuff3}
n_{i+1}< 2n_i \quad\text{holds for infinitely many } i\in\mathbb{N},
\end{equation}
then \eqref{eq:gensuff1} is satisfied too, since
\[ \sum_{j=m_k+1}^{\infty} \frac{1}{n_j} 
> \frac{1}{n_{m_k}} \sum_{j=m_k+1}^{\infty} \frac{1}{2^{j-m_k}} = \frac{1}{n_{m_k}} \]
will be a strict inequality as well, due to the fact that $n_j<2^{j-m_k}n_{m_k}$ then holds for all sufficiently large $j$.
\end{remark}

\section{Proof of Theorem \ref{thm01}}
\label{sec:thm01}

In order to apply Proposition \ref{prop:general}, we need to construct appropriate sequences $n_1,n_2,\ldots$ and $m_1,m_2,\ldots$. As we need $n_{m_k}$ to be divisible by $n_j$ for all indices $j < m_k$, while the inequality $n_{i+1}/n_i \geqslant \lambda$ must hold for all $i \in \mathbb{N}$, the following lemma will come in handy. 

\begin{lemma}\label{lm:divisors}
For every $\lambda\in(1,2)$ and every $Q\in\mathbb{N}$, there exists a positive integer $N$ divisible by $Q$ and for which some subsequence of its divisors
\[ 1=d_{0}<d_{1}<d_{2}<\cdots<d_{M-1}<d_{M}=N \] 
satisfies the inequality
\[ \lambda \leqslant \frac{d_{j+1}}{d_{j}} \leqslant 2 \]
for all $j$ with $0\leqslant j\leqslant M-1$.
\end{lemma}

\begin{proof}
If $Q$ is a power of $2$, i.e., $Q=2^a$ for some $a\in\mathbb{N}\cup\{0\}$, then we can simply take $N:=Q$ together with all its divisors $d_j=2^j$, $0\leqslant j\leqslant a$.
Otherwise, $\log_2 Q$ is an irrational number, which implies that the set
\[ \{ a - b \log_2 Q : a,b\in\mathbb{N} \} \]
is dense in $\mathbb{R}$, so that 
\[ \Big\{ \frac{Q^b}{2^a} : a,b\in\mathbb{N} \Big\} \] 
is dense in $(0,\infty)$. In particular, for a given $1<\lambda<2$, positive integers $a$ and $b$ exist such that $\lambda \leqslant Q^b/2^a\leqslant 2$.
Now we take $N:=2^a Q^b$ and list its desired divisors respectively as
\[ 1, 2^1, 2^2, \ldots, 2^a, Q^b, 2^1 Q^b, 2^2 Q^b, \ldots, 2^a Q^b, \]
proving the lemma.
\end{proof}

Now we are ready to prove our first main result by constructing the desired lacunary sequence.
Let $p_1,p_2,p_3,\ldots$ be the sequence of all primes listed in increasing order.

\begin{proof}[Proof of Theorem \ref{thm01}]
We simultaneously prove parts \ref{ita} and \ref{itb}.
Fix $\lambda\in(1,2)$. For each $k\in\mathbb{N}$, applying Lemma \ref{lm:divisors} to 
$Q = p_1 p_2 \cdots p_k$
provides us with a positive integer $N_k$ divisible by $p_1 p_2 \cdots p_k$, and divisors $d_{k,j}$, $0\leqslant j\leqslant M_k$, with the property
\[ \max\Big\{\lambda,2-\frac{1}{k}\Big\} \leqslant \frac{d_{k,j+1}}{d_{k,j}} \leqslant 2 \quad \text{for } 0\leqslant j\leqslant M_k-1 \]
and such that $d_{k,0}=1$, $d_{k,M_k}=N_k$.
The sequence $n_1,n_2,n_3,\ldots$ and the sequence of indices $m_1, m_2, m_3, \ldots$ are then constructed as follows. We start off with $n_1 = 1$ and $n_2 = 2$, $m_1 = 2$. Now assume, for an integer $k \geqslant  2$, that we have already defined $m_1, m_2, \ldots, m_{k-1}$ and $n_1,n_2,\ldots,n_{m_{k-1}}$ in the previous $k-1$ steps, with $n_{m_{k-1}}$ equal to $N_1N_2\cdots N_{k-1}$ and divisible by $n_j$ for all $j$ with $1 \leqslant j < m_{k-1}$.
We then define $m_k = m_{k-1} + M_k$ and, for $1 \leqslant j \leqslant M_k$, define $n_{m_{k-1} + j} = d_{k,j}n_{m_{k-1}}$. Note that 
\[ n_{m_k} = n_{m_{k-1} + M_k} = d_{k,M_k}n_{m_{k-1}} = N_kn_{m_{k-1}} = N_1N_2\cdots N_{k}.\]
More explicitly, the constructed sequence $(n_i)_{i=1}^{\infty}$ should then look like this:
\begin{align*}
& 1 = n_1, \\
& 2 = n_{m_1} = N_1, \\
& d_{2,1}n_{m_1}, \ d_{2,2}n_{m_1}, \ \ldots, \ d_{2,M_2}n_{m_1} = n_{m_2} = N_1N_2, \\
& d_{3,1}n_{m_2}, \ d_{3,2}n_{m_2}, \ \ldots, \ d_{3,M_3}n_{m_2} = n_{m_3} = N_1N_2N_3, \\
& d_{4,1}n_{m_3}, \ d_{4,2}n_{m_3}, \ \ldots, \ d_{4,M_4}n_{m_3} = n_{m_4} = N_1N_2N_3N_4, \ldots.
\end{align*}

We would like to apply Proposition \ref{prop:general} to this sequence, so we should verify that all conditions are satisfied.
By construction, $n_{m_k} = n_{m_{k-1}} N_k$ is divisible by $n_j$ for all indices $j < m_k$, verifying \ref{it3}.
Moreover, as $n_{m_k}=N_1 N_2 \cdots N_k$ is divisible by $p_1^k p_2^{k-1} \cdots p_{k-1}^2 p_k$, condition \ref{it2} also follows.
Again by construction, we have
\begin{equation}\label{eq:tmpineqe}
\max\Big\{\lambda,2-\frac{1}{k}\Big\} \leqslant \frac{n_{i+1}}{n_i} \leqslant 2
\end{equation}
for every $i\in\mathbb{N}$, where $k$ is the smallest positive integer such that $i<m_k$.
In particular, from \eqref{eq:tmpineqe} we see that $(n_i)_{i=1}^\infty$ is $\lambda$-lacunary and satisfies $\lim_{i\to\infty}n_{i+1}/n_i=2$.
Also, \eqref{eq:gensuff2} holds and Remark \ref{rem:cond} then confirms \ref{it4}.
Note that \ref{it2} and \eqref{eq:gensuff2} together clearly imply \eqref{eq:gensuff3} and thus \eqref{eq:gensuff1} as well, once again by Remark \ref{rem:cond}.
Moreover, note that \eqref{eq:tmpineqe} and $n_1=1$ also give $n_i\leqslant 2^{i-1}$, so that
\[ \sum_{i=1}^{\infty}\frac{1}{n_i} > \sum_{i=1}^{\infty} 2^{-i+1} = 2. \]
The inequality is strict as some of the $n_i$ are clearly not powers of $2$.
Proposition \ref{prop:general} may therefore be applied, and we conclude that for every rational $q \in (0, 2]$ there are infinitely many finite sets $S \subset \mathbb{N}$ with $\sum_{i\in S}1/n_i = q$.

Finally, part \ref{itc} is precisely the content of Corollary \ref{cor:necess1}.
\end{proof}

With the use of integers for which the conclusion of Lemma \ref{lm:divisors} holds, one can turn the above proof into an explicit procedure to create one such infinite sequence $n_1,n_2,\ldots$. To give an example of what such a procedure looks like, let us take $\lambda=3/2$, and work out what the first four steps could be, although there are many options for $N_k$ one may choose. In particular, we do not exactly follow the proof of Lemma \ref{lm:divisors}, but we just take $N_k$ and its divisors to be as simple as possible.

\emph{Step 1.} Start off with $n_1=1$ and $n_2 = N_1 = 2$.

\emph{Step 2.} We set $N_2$ to $2 \cdot 3 = 6$ and take its divisors $2<3<6$. We multiply these divisors by $N_1$, and append the sequence with $4,6,12$.

\emph{Step 3.} Now it is convenient to take $N_3=2^2\cdot 3\cdot 5=60$ with divisors $2<4<6<10<15<30<60$. We append the sequence with these numbers multiplied by $N_1N_2 = 12$, which are $24, 48, 72, 120, 180, 360, 720$.

\emph{Step 4.} This time we can take $N_4=2^5\cdot3^2\cdot 5\cdot 7=10080$ with divisors $2<4<8<16<32<48<72<140<210<315<630<1260<2520<5040<10080$. We append the sequence with these numbers multiplied by $N_1N_2N_3=720$.

After these four steps we arrive at $27$ terms of one possible $(3/2)$-lacunary sequence:
\begin{align*}
& 1, 2, 4, 6, 12, 24, 48, 72, 120, 180, 360, 720, \\
& 1440, 2880, 5760, 11520, 23040, 34560, 51840, 100800, \\
& 151200, 226800, 453600, 907200, 1814400, 3628800, 7257600, \ldots.
\end{align*}
By our construction:
\begin{itemize}
\item sums of reciprocals of the first $2$ terms represent all fractions in $[0,2)$ with denominators dividing $2$;
\item sums of reciprocals of the first $5$ terms represent all fractions in $[0,2]$ with denominators dividing $2^2\cdot 3$;
\item sums of reciprocals of the first $12$ terms represent all fractions in $[0,2]$ with denominators dividing $2^4\cdot 3^2\cdot 5$;
\item sums of reciprocals of the first $27$ terms represent all fractions in $[0,2]$ with denominators dividing $2^9\cdot 3^4\cdot 5^2\cdot 7$.
\end{itemize}
These properties of the first $27$ terms are also easy to verify using rational arithmetic software.

%%%%%

\section{Proof of Theorem \ref{thm02}}
\label{sec:thm02}

In order to prove Theorem \ref{thm02}, the following generalization of Lemma \ref{lm:divisors} will be useful for us.

\begin{lemma}\label{lm:divisors2}
For $\lambda\in(1,2)$ and $Q\in\mathbb{N}$, every given list of positive integers $1=n_1<n_2<\cdots<n_K$ can be extended with positive integers $n_{K+1},\ldots,n_M$ (for some $M\in\mathbb{N}$, $M\geqslant  K$) such that $Q$ and 
\[ 1=n_{1}<n_{2}<\cdots<n_{M-1}<n_{M} \] 
all divide $N:=n_M$, and such that the inequality
\[ \lambda \leqslant \frac{n_{j+1}}{n_{j}} \leqslant 2 \]
is satisfied for all $j$ with $K\leqslant j\leqslant M-1$.
\end{lemma}

\begin{proof}
The claim is contained in Lemma \ref{lm:divisors} for $K=1$, so assume $K\geqslant 2$ and denote $V:=n_1\cdots n_{K-1}Q$.
If $V=2^a$ for some $a\in\mathbb{N}$, then take $N=n_1\cdots n_K Q$ and append the list $n_1,\ldots,n_K$ with divisors
\[ 2^1 n_K, 2^2 n_K, \ldots, 2^a n_K = N. \]
Otherwise, $\log_2 V$ is irrational, so choose $a,b\in\mathbb{N}$ such that $\lambda \leqslant V^b/2^a\leqslant 2$. Then we can take $N:=2^a V^b n_K$ and append $n_1,\ldots,n_K$ with its divisors
\[ 2^1 n_K, 2^2 n_K, \ldots, 2^a n_K, V^b n_K, 2^1 V^b n_K, 2^2 V^b n_K, \ldots, 2^a V^b n_K = N. \qedhere \]
\end{proof}

Having established Lemma \ref{lm:divisors2}, we are now ready to prove Theorem \ref{thm02}.

\begin{proof}[Proof of Theorem \ref{thm02}]
We first work with some fixed $1<\lambda<2$.
Recall the sequence $(a_i)_{i=1}^{\infty}$ defined in the theorem formulation and denote $r:=\sum_{i=1}^{\infty}1/a_i$.
Take a $\lambda$-lacunary sequence of positive integers $n_1<n_2<n_3<\cdots$ and suppose that $P((1/n_i)_{i=1}^{\infty})$ contains all rational numbers from a non-empty interval $(\alpha,\beta)$. Since $P((1/n_i)_{i=1}^{\infty})$ is clearly contained in $[0,\sum_{i=1}^{\infty}1/n_i]$, we have
\[ \beta-\alpha \leqslant \sum_{i=1}^{\infty} \frac{1}{n_i}. \]
Inductively we immediately conclude $n_i\geqslant  a_i$ for every $i\in\mathbb{N}$, so
\[ \beta-\alpha \leqslant \sum_{i=1}^{\infty} \frac{1}{a_i} = r. \]
Thus, $r$ is an upper bound on the lengths of rational intervals filled by $\lambda$-lacunary sequences, i.e., $R(\lambda)\leqslant r$.

Now take $\varepsilon>0$, let $K\in\mathbb{N}$ be an index such that $\sum_{i=1}^{K} 1/a_i > r-\varepsilon$, and simply set $n_i:=a_i$ for $1\leqslant i\leqslant K$.
From
\[ \lambda a_i \leqslant \lceil \lambda a_i\rceil = a_{i+1}\leqslant \lceil 2a_i\rceil = 2a_i \]
we see that 
\begin{equation}\label{eq:auxinthm02p}
\lambda \leqslant \frac{n_{i+1}}{n_i} \leqslant 2
\end{equation}
holds for $1\leqslant i\leqslant K-1$.
Now we apply Lemma \ref{lm:divisors2} to $n_1,\ldots,n_K$ (and with $Q=1$), to obtain the integers $n_{K+1},\ldots,n_M=N$. We deduce that \eqref{eq:auxinthm02p} is satisfied for $K\leqslant i\leqslant M-1$ as well, while $N$ is divisible by $n_i$ for all $i$ with $1\leqslant i< M$.
Then we continue the construction just as in the proof of Theorem \ref{thm01}: in the $k$th step we append the existing terms $n_1,\ldots,n_{m_{k-1}}$ (where $m_0:=M$) with $n_{m_{k-1}}$-multiples of the divisors $d_{k,1},\ldots,d_{k,M_k}=N_k$ of some $N_k$ obtained from Lemma \ref{lm:divisors} applied to $Q=p_1\cdots p_k$. This produces an enlarged list of terms $n_1,\ldots,n_{m_k}$ and completes the $k$th step.
Inequality \eqref{eq:auxinthm02p} continues to hold for all indices and conditions \ref{it2}--\ref{it4} of Proposition \ref{prop:general} are verified exactly as before, so we can conclude that $P((1/n_i)_{i=1}^{\infty})$ contains all rational numbers in $[0,\sum_{i=1}^{\infty}1/n_i)$. Thus,
\[ R(\lambda) \geqslant  \sum_{i=1}^{\infty}\frac{1}{n_i} > \sum_{i=1}^{K}\frac{1}{n_i} = \sum_{i=1}^{K}\frac{1}{a_i} > r-\varepsilon. \]
Since $\varepsilon>0$ was arbitrary, we conclude that $R(\lambda)=r$, finishing the proof of the first claim.

As for the second part of the theorem on certain (limiting) values of $R(\lambda)$, if $1<\lambda<K/(K-1)$ for some integer $K \geqslant  2$, then the previously defined sequence $(a_i)_{i=1}^{\infty}$ starts off with $1,2,\ldots,K$, so
\[ R(\lambda) \geqslant  \sum_{i=1}^{K} \frac{1}{i}, \]
proving that $R(\lambda)\to+\infty$ as $\lambda\to1^+$.
Moreover, using the formula for $R(\lambda)$ in the case $\lambda\in(1,2)$, the trivial inequality
\begin{equation}\label{eq:trivialfora}
a_{i}\geqslant  \lambda^{i-1} 
\end{equation}
gives
\[ R(\lambda) \leqslant \sum_{i=1}^{\infty}\frac{1}{\lambda^{i-1}} = \frac{\lambda}{\lambda-1}. \]
On the other hand, Theorem \ref{thm01} implies $R(\lambda)\geqslant 2$ for all such $\lambda$. These estimates together imply that $R(\lambda)\to2$ as $\lambda\to2^-$.
Finally, the claim about $R(\lambda)$ being $0$ for $\lambda\geqslant 2$ is precisely Corollary \ref{cor:necess1}.
\end{proof}

Simply because of \eqref{eq:trivialfora}, the tail of the series representing $R(\lambda)$ can be upper bounded by
\[ \sum_{i=K+1}^{\infty} \frac{1}{a_i} \leqslant \sum_{i=K+1}^{\infty} \frac{1}{\lambda^{i-1}} = \frac{1}{\lambda^{K-1}(\lambda-1)}. \]
Thus, for $\lambda\in(1,2)$, $\varepsilon\in(0,1)$, and
\[ K = 2 + \Big\lfloor \log_\lambda \frac{1}{\varepsilon(\lambda-1)} \Big\rfloor, \]
we know that the partial sum $\sum_{i=1}^{K}1/a_i$ approximates the true value of $R(\lambda)$ to an error smaller than $\varepsilon$.
This allows us to take, say, $\lambda=3/2$ and compute $R(3/2)$ reliably to $50$ decimal digits as
\[ 2.40694938638836442986564472688463596121152697197900\ldots \]
using $287$ terms of the sequence $(a_i)_{i=1}^{\infty}$.

%%%%%

\section{Proof of Theorem \ref{thm03}}
\label{sec:thm03}

\begin{proof}[Proof of Theorem \ref{thm03}]
Take $\varepsilon\in(0,1)$ such that
\begin{equation}\label{eq:thm3pr1}
\lambda + \varepsilon < 2 \quad\text{and}\quad \frac{1}{\lambda+\varepsilon} + \frac{1}{\Lambda+\varepsilon} > 1,
\end{equation}
which is possible due to the assumptions $\lambda < \Lambda/(\Lambda-1)\leqslant 2$, where one can check that the first of these latter two inequalities is equivalent to $1/\lambda + 1/\Lambda >1$.
Take $U\in\mathbb{N}$ large enough such that
\begin{equation}\label{eq:thm3pr3}
(\lambda+\varepsilon)^{-U-1} < (\Lambda+\varepsilon) \Big(\frac{1}{\lambda+\varepsilon} + \frac{1}{\Lambda+\varepsilon} - 1\Big)
\end{equation}
and denote
\begin{equation}\label{eq:thm3pr2}
m_0 := 2 + \Big\lfloor \log_2\frac{1}{\varepsilon} \Big\rfloor.
\end{equation}

Start defining the sequence $(n_i)_{i=1}^{\infty}$ by simply putting $n_i := 2^{i-1}$ for $1\leqslant i\leqslant m_0$.
This completes the $0$th step of the algorithm, so let us proceed with steps $1,2,3,\ldots$.
Suppose that, after $k-1$ steps, we have already defined the terms $n_1,n_2,\ldots,n_{m_{k-1}}$.
First, append the list with the term $n_{m_{k-1}+1}:=\lfloor\Lambda n_{m_{k-1}}\rfloor+1$, which satisfies
\begin{equation}\label{eq:thm3pr6}
\Lambda < \frac{n_{m_{k-1}+1}}{n_{m_{k-1}}} \leqslant \Lambda + \frac{1}{n_{m_{k-1}}} \leqslant \Lambda + \frac{1}{2^{m_0-1}} \stackrel{\eqref{eq:thm3pr2}}{<} \Lambda + \varepsilon. 
\end{equation}
Then proceed by adding terms $n_{m_{k-1}+2},\ldots,n_{m_{k-1}+U+1}$ defined by the recurrence 
\[ n_{i+1}:=\lceil \lambda n_{i}\rceil, \]
so that 
\begin{equation}\label{eq:thm3pr4}
\lambda \leqslant \frac{n_{i+1}}{n_{i}} < \lambda + \frac{1}{n_{i}} \leqslant \lambda + \frac{1}{2^{m_0}} \stackrel{\eqref{eq:thm3pr2}}{<} \lambda + \varepsilon 
\end{equation}
holds for $m_{k-1}+1\leqslant i\leqslant m_{k-1}+U$.
Finally apply Lemma \ref{lm:divisors2} to the existing terms (i.e., with $K=m_{k-1}+U+1$) and with $Q=p_1\cdots p_k$ the product of the first $k$ primes. This produces terms $n_{m_{k-1}+U+2},\ldots,n_{m_k}$ such that 
\begin{equation}\label{eq:thm3pr5}
\lambda \leqslant \frac{n_{i+1}}{n_i} \leqslant 2
\end{equation}
is valid for $m_{k-1}+U+1\leqslant i\leqslant m_k-1$, and such that $n_{m_k}$ is a multiple of all preceding numbers, completing the $k$th step.

To summarize \eqref{eq:thm3pr6}--\eqref{eq:thm3pr5}, by the construction above we have that the ratio $n_{i+1}/n_i$, for $i=1,2,3,\ldots$, takes values:
\begin{itemize}
\item $2$ for the initial $m_0-1$ indices,
\item $\in(\Lambda,\Lambda+\varepsilon)$ for the next index,
\item $\in[\lambda,\lambda+\varepsilon)$ for the next $U$ indices,
\item $\in[\lambda,2]$ for the next block of indices up to $m_1-1$,
\item $\in(\Lambda,\Lambda+\varepsilon)$ for the next index,
\item $\in[\lambda,\lambda+\varepsilon)$ for the next $U$ indices,
\item $\in[\lambda,2]$ for the next block of indices up to $m_2-1$, etc.
\end{itemize}
We aim to verify
\[ \frac{1}{n_{i}} < \sum_{j=i+1}^{m_k} \frac{1}{n_j} + \frac{1}{n_{m_k}} \] 
for $k\in\mathbb{N}\cup\{0\}$ and $m_{k-1}\leqslant i<m_k$ (the condition $m_{-1}\leqslant i$ being a void one).
However, the only values of $i$ for which this inequality is not an immediate consequence of Remark \ref{rem:cond} are those with $i=m_{k-1}$, $k\in\mathbb{N}$. In that case, the inequality follows from
{\allowdisplaybreaks
\begin{align*}
\sum_{j=1}^{U+1} \frac{1}{n_{i+j}} 
& \stackrel{\eqref{eq:thm3pr4},\eqref{eq:thm3pr6}}{\geqslant } \frac{1}{n_{i}} \frac{1}{\Lambda+\varepsilon} \sum_{j=0}^{U} \frac{1}{(\lambda+\varepsilon)^j} \\
& \quad \ = \frac{1}{n_{i}} \frac{1}{\Lambda+\varepsilon} \frac{1-(\lambda+\varepsilon)^{-U-1}}{1-(\lambda+\varepsilon)^{-1}} \\
& \stackrel{\eqref{eq:thm3pr1},\eqref{eq:thm3pr3}}{>} \frac{1}{n_{i}}. 
\end{align*}
}
Therefore, $(n_i)_{i=1}^{\infty}$ satisfies condition \ref{it4} from Proposition \ref{prop:general}. By construction it also satisfies conditions \ref{it2}--\ref{it3} of the same proposition, while it is clearly $\lambda$-lacunary. Applying Proposition \ref{prop:general} then concludes the theorem.
\end{proof}

Let us comment on the optimality of the lacunarity parameter in Theorem \ref{thm03}. If $(n_i)_{i=1}^{\infty}$ is a $\lambda$-lacunary sequence with $\lambda\geqslant \Lambda/(\Lambda-1)$ and if $i$ is an index such that $n_{i+1}>\Lambda n_i$, then $n_j > \Lambda \lambda^{j-i-1} n_i$ for $j>i$, so that
\begin{align*}
\sum_{j=i+1}^{\infty} \frac{1}{n_j} &< \frac{1}{n_i\Lambda} \sum_{k=0}^{\infty} \frac{1}{\lambda^k} \\
&= \frac{1}{n_i\Lambda} \frac{1}{1-1/\lambda}\\
&\leqslant \frac{1}{n_i},
\end{align*}
 contradicting \eqref{eq:neceint}.
Corollary \ref{cor:necess2} then prevents $P((1/n_i)_{i=1}^{\infty})$ from containing all rational numbers in $[0,\sum_{i=1}^{\infty}1/n_i)$.

Finally, from the proofs of Theorems \ref{thm01}--\ref{thm03} one can see that all requirements imposed on the sequence $(n_i)_{i=1}^{\infty}$ can be combined, if one so desires. Thus, it is possible to formulate one big theorem on the existence of a $\lambda$-lacunary sequence with $\limsup_{i\to\infty} n_{i+1}/n_i = \lambda/(\lambda-1)$ that represents all rational numbers in the interval $(0,R(\lambda)-\varepsilon)$ infinitely many times for a fixed $\varepsilon>0$. We did not think such a formulation would necessarily be productive in order to comprehend these results, however.

%%%%%%%%%%%%%%%%%%%%%%%%%%%%%%%%%%%%%%%%%%%%%%%%

\appendix

\section{Alternative proof of Eppstein's theorem}

Here is a reformulation of \cite[Theorem 1]{Epp21}.

\begin{theorem}\label{thm04}
If $S\subseteq\mathbb{N}$ is a set satisfying $2S\subseteq S$ and containing a multiple of each odd number, then $P((1/n)_{n\in S})=[0,\sum_{n\in S}1/n)\cap\mathbb{Q}$.
\end{theorem}

Even though the set $S$ in the theorem statement does not fulfill the conditions of Proposition \ref{prop:general}, we can nevertheless use a very similar proof technique to reprove this result.

\begin{proof}
Let $(n_i)_{i=1}^\infty$ be the strictly increasing sequence enumerating $S$.
Also, let 
\[ q = \frac{a}{b} < \sum_{i=1}^{\infty} \frac{1}{n_i} \] 
be a positive rational number and let $K\in\mathbb{N}$ be sufficiently large such that $q \leqslant r := \sum_{i=1}^{K} 1/n_i$. Define $L$ as the least common multiple of the integers $n_1, \ldots, n_K, b$ and write $L = 2^{k} \ell$, with $k\in\mathbb{N}\cup\{0\}$ and $\ell$ an odd positive integer. Moreover, let $J,m\in\mathbb{N}$ be such that $n_J$ is divisible by $\ell$ and $2^{m-1} n_K < n_J \leqslant 2^{m} n_K$. Such an index $J$ exists by the assumption that $S$ contains multiples of all odd integers. Denote $M := K+k+2m-1$ and define the sequence $n'_1, \ldots, n'_M$ as
\[ n_1, \ldots, n_K, 2^1 n_K, \ldots, 2^{m-1} n_K, n_J, 2^1 n_J, \ldots, 2^{k+m-1} n_J. \]

First of all, by the assumption $2S \subseteq S$ we see that all integers $n'_i$ we have defined thus far are elements of $S$, with $n'_{i+1}\leqslant 2n'_i$ for all $i \ge 1$. Secondly, $n'_M$ is divisible by $2^{k} \ell = L$, so that $n'_i$ divides $n'_M$ for all $1\leqslant i \leqslant K$. Similarly, one can check that $n'_M$ is divisible by all $n'_i$ with $1\leqslant i<M$, while it is also divisible by $b$. 
Just as in Remark \ref{rem:cond}, the property $n'_{i+1}\leqslant 2n'_i$ implies that the real numbers $x_i:=1/n'_i$ satisfy the condition of Lemma \ref{lm:reals}.
As a consequence, the sums 
\[ \sum_{i\in T}\frac{1}{n'_i}, \quad T\subseteq\{1,\ldots,K+m+M\} \]
$(1/n'_M)$-densely fill in the segment $[0,r]$.
Now we reason similarly as in the proof of Proposition \ref{prop:general}: all of these sums are multiples of $1/n'_M$, so they attain all values from $[0,r]\cap(1/n'_M)\mathbb{Z}$.
Since $q$ can be rewritten as a fraction with denominator equal to $n'_M$, it can therefore be represented as a finite sum 
\[ \frac{1}{n'_{i_1}} + \ldots + \frac{1}{n'_{i_t}} \]
for some indices $i_1<\cdots<i_t$. 
\end{proof}

%%%%%%%%%%%%%%%%%%%%%%%%%%%%%%%%%%%%%%%%%%%%%%%%

\section*{Acknowledgments}
V. K. was supported by the Croatian Science Foundation under the project number HRZZ-IP-2022-10-5116 (FANAP).
The authors are grateful to Thomas Bloom for creating and maintaining the website \emph{Erd\H{o}s problems}, which helped initiate this project. They would also like to thank Sayan Dutta and Zach Hunter for their comments on the website discussion forum. Finally, the authors are indebted to the anonymous referee for their numerous suggestions on improving the exposition.

%%%%%%%%%%%%%%%%%%%%%%%%%%%%%%%%%%%%%%%%%%%%%%%%

\bibliography{ErdosProblem355}{}
\bibliographystyle{plainurl}

\end{document}